\documentclass[10pt, reqno]{amsart}
\usepackage{amsmath,amsfonts,amssymb,amsthm,enumerate,appendix,caption}
\usepackage{cmap,mathtools}
\usepackage{paralist,bm}
\usepackage{graphics} 
\usepackage{epsfig} 
\usepackage{graphicx}
\usepackage{braket}
\usepackage{epstopdf}
 \usepackage[colorlinks=true]{hyperref}
\usepackage{xcolor}
\definecolor{myblue}{rgb}{0.0, 0.0, 1.0}
\definecolor{mygreen}{rgb}{0.01,0.75,0.20}
\hypersetup{ linkcolor=myblue, colorlinks=true, urlcolor=black, citecolor=red}
\usepackage{hyperref}

\newtheorem{theorem}{Theorem}[section]

\newtheorem{thm}[theorem]{Theorem}
\newtheorem{lem}[theorem]{Lemma}

\newtheorem{proposition}[theorem]{Proposition}

\newtheorem{definition}[theorem]{Definition}

\newtheorem{remark}[theorem]{Remark}

\newtheorem{rem}[theorem]{Remark}

\theoremstyle{definition}

\usepackage[margin=1in]{geometry}

\numberwithin{equation}{section}

\newcommand{\dx}{\,\mathrm{d}x}
\newcommand{\dy}{\,\mathrm{d}y}

\DeclarePairedDelimiter\norm{\lVert}{\rVert}%
\makeatletter
\let\oldnorm\norm
\def\norm{\@ifstar{\oldnorm}{\oldnorm*}}

\newcommand{\om} {\omega}
\newcommand{\Om} {\Omega}

\newcommand{\la} {\lambda}

\newcommand\restr[2]{{
  \left.\kern-\nulldelimiterspace 
  #1 
  \right|_{#2} 
  }}

\def\w{{\widetilde w}}

\def\w2{{W^{1,2}_0(\Om)}}
\def\hh2{{H^1_0(\Om)}}

\def\C{{\mathcal C}}

\def\N{{\mathbb N}}
\def\F{{\mathcal F}}

\def\R{{\mathbb R}}

\def\({{\Big(}}
\def\){{\Big)}}

\def\ws2{{\F_{\frac{N}{2}}}}

\def\c1{{\C_c^1}}

\def\d{{\rm d}}

\def\dx{{\rm d}x}
\def\dy{{\rm d}y}

\newcommand{\Hmm}[1]{\leavevmode{\marginpar{\tiny%
			$\hbox to 0mm{\hspace*{-0.5mm}$\leftarrow$\hss}%
			\vcenter{\vrule depth 0.1mm height 0.1mm width \the\marginparwidth}%
			\hbox to
			0mm{\hss$\rightarrow$\hspace*{-0.5mm}}$\\\relax\raggedright #1}}}

\author[U.\ Das]{Ujjal Das}
\address[Ujjal Das]{
	BCAM -- Basque Center for Applied Mathematics,
	48009 Bilbao, Spain. Current address: School of Mathematical Sciences, National Institute of Science Education and Re-
search (NISER), Bhubaneswar, Odisha-752050, India}
\email{\href{mailto:udas@niser.ac.in}{udas@niser.ac.in},\href{mailto:udas@bcamath.org}{udas@bcamath.org}}

\author[L.\ Fanelli]{Luca Fanelli}
\address[Luca Fanelli]{
	Ikerbasque, Basque Foundation for Science,
	48011 Bilbao, Spain,
	\newline \phantom{\quad} \&
	Universidad del Pa\'is Vasco / Euskal Herriko Unibertsitatea,
	48080 Bilbao, Spain,
	\newline \phantom{\quad}\&
	BCAM -- Basque Center for Applied Mathematics,
	48009 Bilbao, Spain}
\email{\href{mailto:luca.fanelli@ehu.es}{luca.fanelli@ehu.eus}}

\author[L.\ Roncal]{Luz Roncal}
\address[Luz Roncal]{
	BCAM -- Basque Center for Applied Mathematics,
	48009 Bilbao, Spain,
	\newline \phantom{\quad} \&
	Ikerbasque, Basque Foundation for Science,
	48011 Bilbao, Spain,
	\newline \phantom{\quad} \&
	Universidad del Pa\'is Vasco / Euskal Herriko Unibertsitatea,
	48080 Bilbao, Spain}
\email{\href{mailto:lroncal@bcamath.org}{lroncal@bcamath.org}}

\date{\today}

\begin{document}
	\title[Quantitative Landis-type result for Dirac operators]{ Quantitative Landis-type result for Dirac operators}
	
	
	
%
%
%
	
	\begin{abstract}
{We study quantitative unique continuation at infinity for Dirac equations with bounded matrix-valued potentials. For the massless Dirac operator $\mathcal{D}_n$ in $\mathbb{R}^n$, we establish a Landis-type estimate showing that the vanishing order of any nontrivial bounded solution of $( \mathcal{D}_n + \mathbb{V} ) \varphi = 0$ satisfies a lower bound of order $\exp(-\kappa R^{2} (\log R)^{2})$ as $|x|=R\to \infty$; the quadratic growth in the exponent is sharp, in view of previous known results. Our proof follows a Bourgain--Kenig type approach based on a Carleman inequality for Dirac operators which relies on a local H\"older regularity result, which we also prove. In two dimension, we obtain improved quantitative estimates under symmetry assumptions on the potential $\mathbb{V}$ and for real-valued solutions. 
 Finally, we also derive qualitative Landis-type results for Dirac equations with decaying potentials, including critical decay rates.}

	\end{abstract}

\maketitle

\section{Introduction}
In this article, we are interested to investigate the quantitative unique continuation properties at infinity of the following Dirac equation:
\begin{align} \label{Eq:Dirac}
    ( \mathcal{D}_n + \mathbb{V} ) \varphi = 0 \ \ \mbox{in} \ \R^n \,,
\end{align}
where $\mathcal{D}_n$ is the massless Dirac operator as defined in Section \ref{Sec:Prelim} and $\mathbb{V}:\R^n \rightarrow \mathbb{C}^{N \times N}$ is a bounded matrix-valued potential on $\R^n$, where the order of the matrix is $N=2^{\frac{n}{2}}$ or $N=2^{\frac{(n+1)}{2}}$ depending on whether $n$ is even or odd.
Let $\mathbb{U} \in H^1_{\operatorname{loc}}(\R^n;\mathbb{C}^{N})$ be a bounded weak solution of the above equation, see Definition \ref{Def:WS}. For $R>0$, we define $$M_R[\mathbb{U}]=\inf_{|x|=R} \left[\sup_{B_1(x)} |\mathbb{U}(y)| \right] \,,$$ 
where $|\mathbb{U}(\cdot)|$ is the Euclidean $\ell^2$-norm on $\mathbb{C}^N$ {{and $B_1(x)$ is the open unit ball centered at $x\in\R^n$}}.
We look for a sharp vanishing  estimate, i.e., a lower bound of $M_R[\mathbb{U}]$ for any nontrivial bounded solution $\mathbb{U}$ of \eqref{Eq:Dirac}. A vanishing estimate gives us a quantitative form of the unique continuation property at infinity of the solutions of \eqref{Eq:Dirac}.    

Such unique continuation results are related to {{a}} well-known {conjecture {{by }}Kondrat'ev} and Landis, {{usually referred to as}} the {\it Landis' conjecture}. In 1988, V. A. Kondrat'ev and E. M. Landis provided an up-to-date survey \cite{KL88} on the most important results of the qualitative theory of second-order elliptic and parabolic equations obtained by various authors over the previous years. In \cite[\S3.5]{KL88} (see also \cite[\S3.5]{KL91}), they raised the question whether any bounded solution of the Schr\"odinger equation:
\begin{equation} \label{Eq:Sch}
    (-\Delta + V)u=0 \ \   \mbox{in} \ \R^n \,, \ \ \ \| V\|_{L^{\infty}(\R^n)} \leq 1 \,,
\end{equation}
which decays as fast as $\mathrm{e}^{-\kappa|x|}$ for some $\kappa > 1$, are trivial. Meshkov \cite{M91} disproved the conjecture by constructing a complex-valued potential $V$ and a nontrivial, complex-valued, bounded solution $u$ of \eqref{Eq:Sch} in $\R^2$ which decays as $\mathrm{e}^{-C |x|^{\frac{4}{3}}}$ for some $C > 0$. In fact, Meshkov completely addressed the conjecture by proving that the exponent $\frac{4}{3}$ is sharp in the sense that, if $u$ decays faster than $\mathrm{e}^{- |x|^{\frac{4}{3}+\varepsilon}}$ for some $\varepsilon >0$, then $u \equiv 0$ in $\R^n$. Motivated by the study of the problem of Anderson localization for the continuous Bernoulli model, Bourgain and Kenig derived a quantitative version of this property \cite{BK05}; using a Carleman-type inequality they obtained the following vanishing order estimate of any bounded normalized solution of \eqref{Eq:Sch}:
\begin{equation}
    \label{Eq:QUCE}
    M_R[u] \gtrsim  \mathrm{e}^{-C R^{\frac{4}{3}}\log R} \  \ \mbox{for some}  \ C>0 \ \ \mbox{independent of} \ R\gg1 \,, 
\end{equation} 
which implies that if $u$ decays faster than $\mathrm{e}^{- \kappa |x|^{\frac{4}{3}}\log |x|}$ for some $\kappa\gg 1$ near infinity, then $u \equiv 0$ in $\R^n$. {Here we write $A\lesssim B$ to indicate that $A\le C B$ with a positive constant $C$ independent of significant quantities and $A\gg B$ means that there exists a big, universal, constant $C$ such that $A\geq C B$}. This ascertain the Landis conjecture with a stronger decay assumption on $u$ than what Landis proposed, but weaker than that of Meshkov. At this point, we would like to {emphasize the} {{most}} fundamental difference between Meshkov's result and that of Bourgain--Kenig's: Meshkov's result is qualitative in nature, i.e., it does not provide any vanishing  estimate of the solution, while Bourgain--Kenig's result gives us a quantitative unique continuation estimate \eqref{Eq:QUCE}. In the literature, we distinguish these two classes of results  by referring them as qualitative and quantitative Landis-type result.  Both qualitative and quantitative results show the limitations of \textit{Carleman's approach}, which does not distinguish between the real and complex case, and this led to rise the question whether Landis' conjecture is valid for the case in which the solution and the potential are real-valued {\cite[Question~1]{K05}, \cite[p. 28]{K07}}. 

A few qualitative Landis-type results are obtained in the real valued case for more general second order, linear, elliptic operator using different approaches such as ODE-techniques \cite{R21}, probabilistic tools \cite{ABG19}, comparison principle \cite{SS21}, and criticality theory \cite{DP24}. Although, some of these results are able to prove the conjecture in all dimension with the decay condition as Landis proposed, but they assume that the underlying operator is nonnegative. {It is also worth mentioning a recent result \cite{FK23} where the equation is considered in a cylinder $\R\times (0,2\pi)^n$ with periodic boundary conditions and the potential is assumed to be bounded and real-valued: it is shown in \cite{FK23} that the fastest rate of decay at infinity of non-trivial solutions is $O(e^{-c|w|})$ for $n=1,2$, and $O(e^{-c|w|^{4/3}})$ for $n\geq 3$, where $w$ is the axial variable. This result suggests that the weak version of Landis' conjecture might not be true in dimensions higher than $3$, and the right decay should match with the complex-valued case.}

On the other hand, there is a significant attention towards the quantitative Landis-type results which concern the improvement of Bourgain--Kenig's estimate \eqref{Eq:QUCE} in the real-valued case.  It seems that obtaining the quantitative Landis-type results are more challenging. To the best of our knowledge, the improved Bourgain--Kenig's estimate {in the real-valued case} are available only in dimension two. In dimension two,  Kenig--Silvestre--Wang \cite{KSW15} reduced the original equation with nonnegative potential $V$ to an inhomogeneous $\bar{\partial}$-equation, and then used the Carleman-type inequality for $\bar{\partial}$-operator along with Hadamard-type three-ball inequality to obtain an  improved  quantitative estimate in $\R^2$. They obtained the estimate \eqref{Eq:QUCE} with $\mathrm{e}^{-CR\log R}$ in the right hand side of it. In \cite{LMNN20}, Logunov--Malinnikova--Nadirashvili--Nazarov further improved the Kenig--Silvestre--Wang's estimate without the nonnegativity assumption on $V$  using quasiconformal mappings and the nodal structure of the solution. Using these approaches, many quantitative unique continuation estimates are derived in recent times for more general second-order elliptic operators, for instance, see \cite{EB23, BS23, DKW17} and the references therein.  

{{Similar}} unique continuation results have been studied in different settings. We refer to \cite{EKPV10} for the time-dependent Schr\"odinger equation, \cite{EKPV16} for the heat equation, {and to \cite{RW19} for a non local version}. Recently, it has also been investigated on certain Riemanian manifold \cite{PPV24} and in the discrete settings \cite{DKP25,FRS25}. {We refer to the nice overview by Malinnikova \cite{Ma23} and} the comprehensive survey \cite{FRS25b} on the Landis conjecture.

In this context, we consider here the Dirac equation \eqref{Eq:Dirac}.  A qualitative Landis-type result for this equation is obtained by Cassano in \cite{Ca22}. He proved that any solution $\mathbb{U}$ of \eqref{Eq:Dirac} for which $\mathrm{e}^{k |x|^{2}} \mathbb{U}$ is  $L^2$-integrable near infinity for all $k \in \mathbb{N}$ must be  compactly supported in $\R^n$. Hence, due to the unique continuation property of the Dirac operator \cite{HP76}, we can conclude that $\mathbb{U} \equiv 0$ in whole $\R^n$.    
Furthermore, Cassano provided examples to show that this decay criterion is sharp.  In this note, our goal is to obtain a quantitative Landis-type results for the Dirac equation \eqref{Eq:Dirac}. Towards this, we follow the Bourgain--Kenig's approach of using the Carleman inequality for Dirac operators and obtain the following result.
\begin{theorem} \label{Thm:main}
    Let $\mathbb{U} \in H^1_{\operatorname{loc}}(\R^n,\mathbb{C}^{N})$ be a bounded solution to \eqref{Eq:Dirac} such that $|\mathbb{U}(0)|=1$ and $\mathbb{V}: \R^n \rightarrow  \mathbb{C}^{N \times N}$   be a bounded matrix-valued potential. Then there exists $\kappa>0$ such that
 \begin{align} \label{Est:Thm1}
     M_R[\mathbb{U}] \gtrsim  \mathrm{e}^{-\kappa R^{2} (\log R)^{2}}
 \end{align}
 for $R\gg1$.
\end{theorem}

It is worth mentioning that the quadratic growth $R^{2}$ in the vanishing order estimate of $M_R[\mathbb{U}]$ is in fact sharp due to the result in {\cite[Theorem 1.5]{Ca22}}. Observe that if $\mathbb{U}$ decays as $\mathrm{e}^{-|x|^{2+\varepsilon}}$ near infinity for some $\varepsilon>0$, then by the above theorem $\mathbb{U} \equiv 0$ in $\R^n$. We do not need the unique continuation property to conclude this. We would like to highlight that our proof of above theorem following Bourgain-Kenig's approach requires the H\"older regularity of the solution $\mathbb{U}$, which is proved in Theorem \ref{thm:holder}. This regularity result is also of independent interest in its own right. 

In view of our earlier discussion on the improved Bourgain-Kenig's estimate for the real-valued solutions of  the Schr\"odinger equation \eqref{Eq:Sch}, it is reasonable to look for certain class of $\mathbb{V}, \mathbb{U}$ for which we can improve the vanishing order estimate \eqref{Est:Thm1} of the solution of the Dirac equation \eqref{Eq:Dirac}. In Theorem \ref{Thm:pot}, we obtain an improved  estimate of $M_R[\mathbb{U}]$ in two dimension under some symmetry assumptions on $\mathbb{V}$. While, in Theorem \ref{Thm:Real}, we found the same for real-valued solutions of \eqref{Eq:Dirac} in two dimension. {The latter case is relevant within the context of, e.g., Majorana equations.}

Finally, we also consider  bounded potential $\mathbb{V}$ that has some decay, say, $|\mathbb{V}(x)| \lesssim |x|^{-\la}$ near infinity for some $\lambda >0$. In this case, we are interested in deriving a qualitative Landis-type result capturing this decay behavior of the potential. The case $\lambda <1$ {was investigated in \cite{Ca22}, hence allowing $\lambda$ to be negative,} in which case the potential is not bounded. We complement the result in \cite{Ca22} by addressing the case $\lambda\geq 1$ in Theorems \ref{Thm:qual1}, \ref{Thm:qual2}. 

{A remark about notation: along the paper, we will not use different notations for scalar, vector, or matrix quantities. E.g., we will denote by $0$ either the scalar $0$, or the null vector in $\R^n$, or the null vector in $\mathbb{C}^N$. Each case will be clear from the context and from the definition of the functions and variables involved.}

\section{Preliminaries}\label{Sec:Prelim}
In this section, we recall some properties of the Dirac operators, notion of weak solution, some regularity results, and a Carleman inequality for Dirac operators.

\subsection{Dirac operator} The $n$-dimensional mass-less Dirac operator $\mathcal{D}_n$ is a first order differential operator defined by $$\mathcal{D}_n^2=(-\Delta)I_{N \times N} \,,$$
where $N=2^{\frac{n}{2}}$ or $N=2^{\frac{(n+1)}{2}}$ depending on whether $n$ is even or odd. In two and three dimension, we can write it in terms of the {\it Pauli matrices}
\begin{align*}
\sigma_1:=  \begin{pmatrix}
0 & 1 \\
1 & 0 
\end{pmatrix} \,, \ \   \sigma_2:=  \begin{pmatrix}
0 & -i \\
i & 0 
\end{pmatrix} \,, \ \   \sigma_3:=  \begin{pmatrix}
1 & 0 \\
0 & -1 
\end{pmatrix} \,.
\end{align*}
The two dimensional Dirac operator can be written as follows:
\begin{align*}
    \mathcal{D}_2 = -i \sigma_1 \partial_1 -i \sigma_2 \partial_2 = \begin{pmatrix}
0 & -2 i \partial_z \\
-2 i \partial_{\bar{z}} & 0 
\end{pmatrix} \,,
\end{align*}
where $2\partial_z := \partial_x - i \partial_y$, $2\partial_{\bar{z}} := \partial_x + i \partial_y$ and we use the identification $(x,y) \in \mathbb{R}^2$ with $z=x+iy \in \mathbb{C}$. The complex conjugate of $z$ is denoted by $\bar{z}:=x-iy$. In three dimension, the Dirac operator is represented as below:
\begin{align*}
    \mathcal{D}_3= -i \sum_{j=1}^3 \alpha_j \  \partial_j   \,,  \ \ \ \ \ \alpha_j:=\begin{pmatrix}
0 & \sigma_j \\
\sigma_j & 0 
\end{pmatrix} \,.
\end{align*}
In general dimension, $$\mathcal{D}_n:= -i\sum_{j=1}^n \alpha_j \partial_j \,,$$
where $\alpha_j$ are $N \times N$ Hermitian constant matrices of order $N$ and  $\alpha_j$ satisfy the Clifford's anticommutation rule
$$ \alpha_j \alpha_k + \alpha_k \alpha_j = 2 \delta_{jk}I_{N \times N}, $$
for any $j,k=1,2,....n$.
Note that the Dirac operator $\mathcal{D}_n$ is a self adjoint operator on $L^2(\R^n,\mathbb{C}^{N})$ with the domain $H^1(\R^n,\mathbb{C}^{N})$.

\begin{definition}[Weak Solution] \label{Def:WS}
Let $\mathbb{V}:\R^n \rightarrow \mathbb{C}^{N \times N}$ be a bounded matrix-valued potential on $\mathbb{R}^n$. Then $\mathbb{U} \in L^2_{\operatorname{loc}}(\R^n;\mathbb{C}^N)$ is said to be a weak solution of $$(\mathcal{D}_n + \mathbb{V}) \varphi =0 \ \ \mbox{in} \ \R^n$$
if and only if
$$
\int_{\R^n} \big( \langle \mathbb{U}, \mathcal{D}_n \phi \rangle_{\mathbb{C}^N} + \langle \mathbb{U},  \mathbb{V}^*\phi\rangle_{\mathbb{C}^N} \big) \ \dx =0 
$$
i.e., 
$$ 
i\sum_{j=1}^n \int_{\R^n} \mathbb{U}\cdot \big(\bar{\alpha}_j\, \partial_j \bar{\phi}\big) \dx = - \int_{\R^n} \big(\mathbb{V}\mathbb{U}\big)\cdot \bar{\phi} \dx
$$
for all $\phi \in C_c^{\infty}(\R^n; \mathbb{C}^N)$. Here $\mathbb{V}^*$ denotes the conjugate transpose of $\mathbb{V}$, and $\langle \cdot,\cdot\rangle_{\mathbb{C}^N}$ represents the standard inner-product in $\mathbb{C}^N$ and $A \cdot B:= \sum_{i=1}^N A_iB_i$ is the usual dot product for elements $A,B \in \mathbb{C}^N$.
\end{definition}

\begin{remark} \rm 
We may equivalently define the weak solution by taking $\mathbb{U} \in H^1_{\operatorname{loc}}(\R^n;\mathbb{C}^N)$ {{a priori}} because of    the fact that $\mathbb{U} \in L^2_{\operatorname{loc}}(\R^n;\mathbb{C}^N)$ and $\mathcal{D}_n\mathbb{U} \in L^2_{\operatorname{loc}}(\R^n;\mathbb{C}^N)$ together implies $\mathbb{U} \in H^1_{\operatorname{loc}}(\R^n;\mathbb{C}^N)$. This follows from the following inequality:
$$\|\nabla \phi\|_{L^2(\R^n)} \lesssim \|\mathcal{D}_n \phi\|_{L^2(\R^n)} $$
for all $\phi \in C_c^{\infty}(\R^n; \mathbb{C}^N)$, see {Lemma \ref{lem:globalLp}} below.
\end{remark}


\subsection{Carleman inequality}
We recall the main ingredient of Bourgain--Kenig's approach of obtaining a vanishing order estimate, namely the Carleman inequality for Dirac operators. Carleman inequalities for Dirac operators go back to Jerison \cite{J86}, where they were used to derive unique continuation properties for the Dirac and Laplace operators. The version we need here involves the weight $\mathrm{e}^{\tau(\log|x|)^2}$. For a proof we refer to  \cite[(4.7)]{En15}.
\begin{proposition}  \cite[(4.7)]{En15} \label{Prop:CI} For $\tau > 0$, we have
\begin{align} \label{Eq:CI}
    \tau \int_{\R^n}  \frac{\mathrm{e}^{\tau (\log |x|)^2}}{|x|^2} |u(x)|^2 \dx \leq \int_{\R^n}\mathrm{e}^{\tau (\log |x|)^2} |\mathcal{D}_n u|^2 \dx
\end{align}
 for all $u \in C^{\infty}_c(\R^n \setminus \{0\},\mathbb{C}^{N})$.   
\end{proposition}

\subsection{Quantitative unique continuation estimate for $\bar{\partial}$-equation}
Observe that the Dirac equation \eqref{Eq:Dirac} in two dimension reduces to a system of ODE as follows 
\begin{align} \label{Eq:sys}
 i \bar{\partial} \mathbb{U}_1  = V_{21} \mathbb{U}_1  +  V_{22} \mathbb{U}_2,  \qquad
    i \partial \mathbb{U}_2 = V_{11} \mathbb{U}_1  +  V_{12} \mathbb{U}_2 \,,
\end{align}
where $\mathbb{U}=(\mathbb{U}_1,\mathbb{U}_2)$, {$\mathbb{V}=(V_{ij})_{2\times 2}$}, and $\partial:= \partial_x-i\partial_y= 2\partial_{{z}}$ and $\bar{\partial}:= \partial_x+i\partial_y=2\partial_{\bar{z}}$. In this view, it is natural to anticipate that the quantitative unique continuation estimate for the so-called $\bar{\partial}$-equation may help us in deriving the same for the Dirac equation.
The next proposition is due to Kenig--Silvestre--Wang \cite[Theorem 2.1]{KSW15}, which will play the key role for obtaining improved estimates in Theorems \ref{Thm:pot}, \ref{Thm:Real}. 
\begin{proposition}  \label{Prop:KSW}{\cite[Theorem 2.1]{KSW15}}
Let $u$ be any solution of 
$$\bar{\partial}u = Vu \ \ \mbox{in} \  \R^2,$$
where $V \in L^{\infty}(\R^2; \mathbb{C}) $. Assume that {$|u(z)| \leq  \mathrm{e}^{C|z|}$ for some $C >0$ and $|u(0)|=1$. Then there exists $\kappa=\kappa(C, \|V\|_{L^{\infty}(\R^2)})>0$} such that
\begin{align} \label{Eq:KSW_Est}
    M_R[u] \gtrsim \mathrm{e}^{-\kappa R \log R}
\end{align}
for $R\gg 1$.

\end{proposition}

\section{H\"older regularity}
In this section, we prove the H\'older regularity of the solution of the Dirac equation \eqref{Eq:Dirac}, which is required for our main proof. This regularity result is also of independent interest; for this reason, we establish the H\'older regularity for general Dirac operator with mass. 
Define the massive Dirac operator by
$$
\mathcal{D}_{n,m} := \mathcal{D}_n + m \beta,\qquad m \in\R \,,
$$
where $\beta\in\mathbb{C}^{N\times N}$ is a constant Hermitian matrix satisfying the relation 
$$
\alpha_j\beta+\beta\alpha_j=0,\qquad
\beta^2=I.
$$
{As usual, given an open set $\Omega \subseteq \R^n$, the H\"older seminorm of $\psi$ is given by 
$$
[\psi]_{C^{0,\alpha}(\Omega)}:=\sup_{x,y\in \Omega,\, x\neq y}\frac{|\psi(x)-\psi(y)|}{\|x-y\|^\alpha}.
$$}


\begin{thm}[Local H\"older regularity]\label{thm:holder}
Let $n\geq 2$, $m\in\R$, and $\mathbb{V} \in L^\infty(\R^n;\mathbb{C}^{N\times N})$.
If $\psi\in L^2_{\mathrm{loc}}(\R^n;\mathbb{C}^N)$ is a distributional solution of 
\begin{equation}\label{eq:main}
(\mathcal{D}_{n,m} + \mathbb{V})\psi = 0 \quad\text{in }\R^n,
\end{equation}
then $\psi\in C^{0,\alpha}_{\mathrm{loc}}(\R^n;\mathbb{C}^N)$ for some $\alpha\in(0,1)$.
More precisely, for every ball $B_{2R}(x_0)\subset\R^n$ and every $p\in(1,\infty)$,
$$
\|\psi\|_{W^{1,p}(B_R(x_0))}
\le C\,(1+|m|+\|V\|_\infty+R^{-1})\,\|\psi\|_{L^p(B_{2R}(x_0))},
$$
where $C=C(n,p)$ is independent of $x_0,R,\psi$.
Consequently, choosing any $p>n$ and setting $\alpha=1-\frac{n}{p}$, we have
$$\psi\in C^{0,\alpha}(B_R(x_0))$$
and 
$$ [\psi]_{C^{0,\alpha}(B_R(x_0))}\le C'(n,p)\,(1+|m|+\|V\|_\infty+R^{-1})\,\|\psi\|_{L^p(B_{2R}(x_0))} \,.
$$
\end{thm}

The proof has {three parts:
(i) a global Calder\'on--Zygmund estimate (ellipticity in the symbol sense) for the massless Dirac operator; (ii) a local $W^{1,p}$ estimate for the massive Dirac operator;
(iii) a bootstrap in $p$ using Sobolev/Morrey embeddings.}
\begin{proof}[Proof of Theorem \ref{thm:holder}] \noindent
\textbf{Step 1: A global Calder\'on--Zygmund estimate for $\mathcal{D}_{n}$.}
\begin{lem}[Global $L^p$ estimate for compactly supported spinors]\label{lem:globalLp}
Let $1<p<\infty$ and $\varphi\in C_c^\infty(\R^n;\mathbb{C}^N)$.
Then
\begin{equation}\label{eq:globalCZ}
\|\nabla \varphi\|_{L^p(\R^n)} \le C_{n,p}\,\|\mathcal{D}_{n} \varphi\|_{L^p(\R^n)}.
\end{equation}
\end{lem}

\begin{proof}
Take Fourier transforms (componentwise). The symbol of $\mathcal{D}_{n}$ is, for {$\xi=(\xi_1,\ldots, \xi_n)$, $\xi_j\in \mathbb{R}$,}
\[
\sigma_{\mathcal{D}_{n}}(\xi)=\alpha\cdot\xi:=\sum_{j=1}^n\alpha_j\xi_j.
\]
By Clifford relations, for all $\xi\neq 0$, 
\[
(\alpha\cdot\xi)^2 = |\xi|^2 I \quad\Longrightarrow\quad
(\alpha\cdot\xi)^{-1} = \frac{\alpha\cdot\xi}{|\xi|^2}.
\]
Thus
\[
\widehat{\varphi}(\xi)= (\alpha\cdot\xi)^{-1}\,\widehat{\mathcal{D}_{n}\varphi}(\xi)
=\frac{\alpha\cdot\xi}{|\xi|^2}\,\widehat{\mathcal{D}_{n}\varphi}(\xi).
\]
For each $k=1,\dots,n$,
\[
\widehat{\partial_k\varphi}(\xi)= i\xi_k\,\widehat{\varphi}(\xi)
= i\xi_k\,\frac{\alpha\cdot\xi}{|\xi|^2}\,\widehat{\mathcal{D}_{n}\varphi}(\xi)
=: M_k(\xi)\,\widehat{\mathcal{D}_{n}\varphi}(\xi).
\]
The multiplier $M_k(\xi)$ is a matrix-valued function which is smooth on $\R^n\setminus\{0\}$,
homogeneous of degree $0$, and satisfies standard Mihlin bounds 
$$
|\partial_\xi^\gamma M_k(\xi)|\le C_\gamma\,|\xi|^{-|\gamma|}\qquad(\xi\neq 0)
$$
for each multi-index $\gamma$ of order $|\gamma|\leq \lfloor n/2\rfloor+1$.
Hence, by the (vector-valued) Mihlin multiplier theorem,
the operator $T_k$ defined by $\widehat{T_k f}=M_k\widehat{f}$ is bounded on $L^p(\R^n)$.
Applying this to $f=\mathcal{D}_{n}\varphi$ yields
\[
\|\partial_k\varphi\|_{L^p}\le C_{n,p}\,\|\mathcal{D}_{n}\varphi\|_{L^p}.
\]
Summing over $k$ gives \eqref{eq:globalCZ}.
\end{proof}

\begin{remark} \rm
This is the precise sense in which the (massless) Dirac operator is elliptic:
$\sigma_{\mathcal{D}_{n}}(\xi)$ is invertible for $\xi\neq 0$.
It is not a matter of positivity/coercivity.
\end{remark}
\medskip
\noindent
\textbf{Step 2: Local $W^{1,p}$ estimate for solutions of $\mathcal{D}_{n,m} \psi = F$.}

\begin{proposition}[Local $W^{1,p}$ estimate]\label{prop:localWp}
Let $1<p<\infty$ and let $\psi\in L^p_{\mathrm{loc}}(\R^n;\mathbb{C}^N)$ satisfy
\begin{equation}\label{eq:inhom}
\mathcal{D}_{n,m}\psi = F 
\end{equation}
in the sense of distribution, where $F\in L^p_{\mathrm{loc}}(\R^n;\mathbb{C}^N)$.
Then for every ball $B_{2R}(x_0)$ one has
\begin{equation}\label{eq:localWp}
\|\psi\|_{W^{1,p}(B_R(x_0))}
\le C_{n,p}\Big(\|F\|_{L^p(B_{2R}(x_0))} + (1+|m|+R^{-1})\|\psi\|_{L^p(B_{2R}(x_0))}\Big).
\end{equation}
\end{proposition}

\begin{proof}
Fix $x_0$ and $R>0$. Choose $\eta\in C_c^\infty(B_{2R}(x_0))$ such that
$\eta\equiv 1$ on $B_R(x_0)$ and $|\nabla\eta|\le C/R$.
Set $\varphi:=\eta\psi$ (so $\varphi\in L^p$ with compact support).
Compute derivatives (in distributional sense) using $\mathcal{D}_{n}(\eta\psi)=\eta\,\mathcal{D}_{n}\psi -i(\alpha\cdot\nabla\eta)\psi$:
\[
\mathcal{D}_{n}\varphi
= \mathcal{D}_{n}(\eta\psi)
= \eta(\mathcal{D}_{n}\psi) - i(\alpha\cdot\nabla\eta)\psi.
\]
Since $\mathcal{D}_{n,m}=\mathcal{D}_{n}+m\beta$, equation \eqref{eq:inhom} gives $\mathcal{D}_{n}\psi = F - m\beta\psi$,
hence
\[
\mathcal{D}_{n}\varphi = \eta F - m\eta\beta\psi - i(\alpha\cdot\nabla\eta)\psi.
\]
Now apply Lemma \ref{lem:globalLp}\footnote{Actually, apply an extension, by density, of the lemma to compactly supported functions in $L^p$} to $\varphi$:
\[
\|\nabla\varphi\|_{L^p(\R^n)} \le C_{n,p}\,\|\mathcal{D}_{n}\varphi\|_{L^p(\R^n)}.
\]
Because $\eta\equiv 1$ on $B_R(x_0)$, 
\begin{align*}
    \|\nabla\psi\|_{L^p(B_R(x_0))}  \le \|\nabla(\eta\psi)\|_{L^p(\R^n)} + \|\psi\nabla\eta\|_{L^p(\R^n)}  \le \|\nabla\varphi\|_{L^p} + C R^{-1}\|\psi\|_{L^p(B_{2R}(x_0))} \,.
\end{align*}
Estimating $\|\mathcal{D}_{n}\varphi\|_{L^p}$ using the previous identity yields
\[
\|\mathcal{D}_{n}\varphi\|_{L^p}
\le \|F\|_{L^p(B_{2R}(x_0))}
+ |m|\|\psi\|_{L^p(B_{2R}(x_0))}
+ C R^{-1}\|\psi\|_{L^p(B_{2R}(x_0))}.
\]
Combining the last three displays gives \eqref{eq:localWp} (just add $\|\psi\|_{L^p(B_R)}$ to the left and use $\|\psi\|_{L^p(B_R)}\le \|\psi\|_{L^p(B_{2R})}$).
\end{proof}

\medskip
\noindent
\textbf{Step 3: Apply to $(\mathcal{D}_{n,m}+\mathbb{V})\psi=0$ and bootstrap.}

Fix a ball $B_{2R}(x_0)\subset\R^n$.
Rewrite \eqref{eq:main} as
\[
\mathcal{D}_{n,m}\psi = -\mathbb{V}\psi =: F.
\]
If $\psi\in L^p(B_{2R}(x_0))$, then $F\in L^p(B_{2R}(x_0))$ and
\[
\|F\|_{L^p(B_{2R}(x_0))}\le \|\mathbb{V}\|_{L^{\infty}(\R^n)}\,\|\psi\|_{L^p(B_{2R}(x_0))}.
\]
Plugging into Proposition \ref{prop:localWp} yields the key estimate:
\begin{equation}\label{eq:keyWp}
\|\psi\|_{W^{1,p}(B_R(x_0))}
\le C_{n,p}\,(1+|m|+\|V\|_\infty + R^{-1})\,\|\psi\|_{L^p(B_{2R}(x_0))}.
\end{equation}

\medskip\noindent
\textbf{Bootstrap in $p$.}
We start from the assumption $\psi\in L^2_{\mathrm{loc}}$, hence $\psi\in L^2(B_{2R}(x_0))$.
Apply \eqref{eq:keyWp} with $p=2$ to get $\psi\in W^{1,2}(B_R(x_0))$.

If $n>2$, Sobolev embedding gives
\[
W^{1,2}(B_R)\hookrightarrow L^{2^*}(B_R),\qquad 2^*=\frac{2n}{n-2},
\]
so $\psi\in L^{2^*}(B_R)$. Shrink radii (e.g.\ replace $R$ by $R/2$) and apply \eqref{eq:keyWp} with $p=2^*$ to get
$\psi\in W^{1,2^*}(B_{R/2})$, hence by Sobolev embedding $\psi\in L^{(2^*)^*}(B_{R/2})$.
Iterating (with radii $R, R/2, R/4,\dots$) yields
\[
\psi\in W^{1,p}(B_{R/2}(x_0))\quad\text{for every }p<\infty.
\]

If $n=2$, one uses that $W^{1,2}$ embeds into $L^q$ for every $q<\infty$ on bounded domains.
Thus from $\psi\in W^{1,2}(B_R)$ we get $\psi\in L^q(B_R)$ for any finite $q$,
and then \eqref{eq:keyWp} implies $\psi\in W^{1,q}(B_{R/2})$ for any $q<\infty$.

\medskip\noindent
\textbf{H\"older continuity.}
Choose any $p>n$. From the bootstrap we have $\psi\in W^{1,p}(B_{R/2}(x_0))$.
By Morrey's inequality,
\[
W^{1,p}(B_{R/2}(x_0)) \hookrightarrow C^{0,\alpha}(B_{R/2}(x_0)),
\qquad \alpha=1-\frac{n}{p}\in(0,1),
\]
and
\[
[\psi]_{C^{0,\alpha}(B_{R/2}(x_0))}
\le C(n,p)\,\|\psi\|_{W^{1,p}(B_{R/2}(x_0))}.
\]
Combining with \eqref{eq:keyWp} (with radii adjusted) gives the stated H\"older estimate.
Since $x_0$ and $R$ were arbitrary, $\psi\in C^{0,\alpha}_{\mathrm{loc}}(\R^n)$.
\end{proof}

\begin{rem} \rm
(1) The argument never differentiates $\mathbb{V}$, so $\mathbb{V} \in L^\infty$ is enough.
Squaring the equation would produce terms involving the partial derivatives of $ \mathbb{V}_{ij}$, which is why one avoids that route.


(2) By taking $p$ arbitrarily large, one gets $\psi\in C^{0,\alpha}_{\mathrm{loc}}$ for every $\alpha<1$,
but not necessarily Lipschitz without stronger assumptions.
\end{rem}

\section{Main Results}\label{Sec3}



\subsection{Quantitative estimates in general situation} \, \\
In this sub-section, we prove our main result Theorem \ref{Thm:main}. Our proof is based on the Carleman inequality for Dirac operator (Proposition \ref{Prop:CI}) following Bourgain--Kenig's approach \cite{BK05}. 

\begin{proof}[Proof of Theorem \ref{Thm:main}]
Let $\mathbb{U} \in H^1_{\operatorname{loc}}(\R^n,\mathbb{C}^{N})$ be a bounded solution to \eqref{Eq:Dirac} such that $|\mathbb{U}(0)|={ 1}$. Recall that we need to find a lower estimate of   $$M_R[\mathbb{U}]:=\inf_{|x|=R} \left[\sup_{y \in B_1(x)} |\mathbb{U}(y)| \right] $$ 
for $R\gg 1$. 
Fix $x_0 \in \R^n$ arbitrary with $|x_0|=R$ for $R\gg 1$. Interchanging $0$ and $x_0$ and rescaling, we define 
$$\widetilde{\mathbb{U}}(x)=\mathbb{U}\Big(AR\big(x+\frac{x_0}{AR}\big)\Big),$$
where $A>0$ is large dimensional constant that will be chosen later.
Note that $|\widetilde{\mathbb{U}}(\tilde{x}_0)|={ 1}$ for $\tilde{x}_0=-\frac{x_0}{AR}$, $|\tilde{x}_0|=\frac{1}{A}$, and set 
$$M_R[x_0]:= \sup_{B_{1}(x_0)} |{\mathbb{U}}| =\sup_{B_{r_0}} |\widetilde{\mathbb{U}}| \,,$$
where $r_0=\frac{1}{AR}$. 
We are going to apply the Carleman inequality \eqref{Eq:CI} with $u={\rho} \widetilde{\mathbb{U}}$, where the cut-off function $\rho: \R^n \rightarrow \R$ is taken to be a smooth function satisfying the following:
\begin{align}
\rho \equiv \begin{cases}
    0 \,, \ \ \mbox{on} \ \ |x| \leq \frac{r_0}{4} \ \ \mbox{and} \ \ |x|>4 \\
    1  \,, \ \ \mbox{on} \ \ \frac{1}{3}r_0 \leq |x| \leq 3 
\end{cases}
    \end{align}
and first order partial derivatives of $\rho$ are bounded by a constant if $|x|>3$ and by a constant times $AR$ if $|x|< \frac{r_0}{3}$. 

We estimate the right hand side of the Carleman inequality \eqref{Eq:CI} as follows
\begin{align}
    I_R & := \int_{\R^n}\mathrm{e}^{\tau (\log |x|)^2} |\mathcal{D}_n (\rho \widetilde{\mathbb{U}})|^2 \dx  \nonumber \\
   & \lesssim  \int_{\R^n}\mathrm{e}^{\tau (\log |x|)^2} \left[\rho^2 |\mathcal{D}_n  \widetilde{\mathbb{U}}|^2 +  |\widetilde{\mathbb{U}}|^2 |\nabla {\rho} |^2\right] \dx  \nonumber \\
   & \lesssim  A^2 R^2 \int_{\R^n}\mathrm{e}^{\tau (\log |x|)^2}  \rho^2  |\widetilde{\mathbb{V}}\widetilde{\mathbb{U}}|^2 \dx  + \int_{\R^n}\mathrm{e}^{\tau (\log |x|)^2}  |\widetilde{\mathbb{U}}|^2 |\nabla {\rho} |^2 \dx \nonumber \\
    & \lesssim A^2 R^{2} \int_{\R^n}\mathrm{e}^{\tau (\log |x|)^2}    |\rho \widetilde{\mathbb{U}}|^2 \dx  +  I_{R,1} \nonumber
    \\
    & \lesssim A^2 R^{2} \int_{\R^n} \frac{\mathrm{e}^{\tau (\log |x|)^2} }{|x|^2}    |\rho \widetilde{\mathbb{U}}|^2 \dx  +  I_{R,1} \label{Est1} \,,
\end{align}
where  $\widetilde{\mathbb{V}}(x):= \mathbb{V}(AR(x+\frac{x_0}{AR}))$ {and the implicit constant in the last estimate depends on $\|\mathbb{V}\|_{L^{\infty}}$}. For the last estimate we use the support of $\rho$.  Now, 
\begin{align*}
    I_{R,1} & := \int_{\R^n}\mathrm{e}^{\tau (\log |x|)^2}  |\widetilde{\mathbb{U}}|^2 |\nabla {\rho} |^2 \dx \\
    & = \int_{\frac{r_0}{4} \leq |x| \leq \frac{r_0}{3}}\mathrm{e}^{\tau (\log |x|)^2}  |\widetilde{\mathbb{U}}|^2 |\nabla {\rho} |^2 \dx + \int_{3 \leq |x| \leq 4}\mathrm{e}^{\tau (\log |x|)^2}  |\widetilde{\mathbb{U}}|^2 |\nabla {\rho} |^2 \dx\\
    & \lesssim  A^2 R^2 \int_{\frac{r_0}{4} \leq |x| \leq \frac{r_0}{3}}\mathrm{e}^{\tau (\log |x|)^2}  |\widetilde{\mathbb{U}}|^2  \dx + \int_{3 \leq |x| \leq 4} \mathrm{e}^{\tau (\log |x|)^2}  |\widetilde{\mathbb{U}}|^2  \dx \\
    & \lesssim A^2R^2 M_R[x_0]^2 \ \mathrm{e}^{\tau (\log (4 AR))^2} \Big(\frac{1}{A^nR^n} \Big) + \mathrm{e}^{\tau (\log 4)^2} \,,
\end{align*}
where we use the fact that $\mathbb{U}$ is bounded in order to estimate the second integral on the right hand side of the above inequality.
Using the above estimate of $I_{R,1}$ in \eqref{Est1}, we obtain 
\begin{align}
  I_R \lesssim  A^2 R^{2} \int_{\R^n} \frac{\mathrm{e}^{\tau (\log |x|)^2}}{|x|^2} \   |\rho \widetilde{\mathbb{U}}|^2 \dx +  A^{2-n} R^{2-n} M_R[x_0]^2 \ \mathrm{e}^{\tau (\log (4 AR))^2} + \mathrm{e}^{\tau (\log 4)^2} \,. 
\end{align}

Observe that by choosing $\tau \sim A^2 R^{2}  $, we can absorb the first integral in the right hand side of the above equation with  the left hand side of the Carleman inequality, which is
\begin{align*}
    I_L & := \tau \int_{\R^n}  \frac{\mathrm{e}^{\tau (\log |x|)^2}}{|x|^2}   |\rho \widetilde{\mathbb{U}}|^2 \dx \,.
    \end{align*}
Consequently, we have
\begin{align} \label{Est2}
    I_L \lesssim A^{2-n} R^{2-n} M_R[x_0]^2 \ \mathrm{e}^{\tau (\log (4 AR))^2} + \mathrm{e}^{\tau (\log 4)^2} \,,
\end{align}
provided $\tau \sim A^2 R^{2} $.
Next we estimate $I_L$ from below as follows:
\begin{align} \label{EqIL}
    I_L 
    & \geq \tau \int_{|x-\tilde{x}_0| < \frac{r_0}{2\gamma}}  \frac{\mathrm{e}^{\tau (\log |x|)^2}}{|x|^2}   |\rho \widetilde{\mathbb{U}}|^2 \dx \,,
    \end{align}
where $\gamma>0$ is large enough such that
\begin{align}
    |\widetilde{\mathbb{U}}(x) - \widetilde{\mathbb{U}}(\tilde{x}_0)| \leq \frac{\gamma^{\alpha}}{2^{1-\alpha}} A^{\alpha} R^{\alpha} |x-\tilde{x}_0|^{\alpha} < \frac{1}{2} \ \ \mbox{if} \ \ |x-\tilde{x}_0|< \frac{r_0}{2\gamma} \,.
\end{align}
Above we use the fact that $\mathbb{U}$ is locally $\alpha$-H\'older continuous for some $\alpha \in (0,1)$ {(Theorem \ref{thm:holder})}. Further, observe that we use the $\alpha$-H\'older continuity of $\mathbb{U}$ at the origin.
If $|x-\tilde{x}_0|< \frac{r_0}{2\gamma}$, then
$$ 1=|\widetilde{\mathbb{U}}(\tilde{x}_0)|\leq |\widetilde{\mathbb{U}}(\tilde{x}_0)-\widetilde{\mathbb{U}}(x)| + |\widetilde{\mathbb{U}}(x)| \leq \frac{1}{2} + |\widetilde{\mathbb{U}}(x)| \,. $$
Hence, $|\widetilde{\mathbb{U}}(x)| \geq \frac{1}{2}$ if  $|x-\tilde{x}_0|< \frac{r_0}{2\gamma}$. For large $R$, if  $|x-\tilde{x}_0|< \frac{r_0}{2\gamma}$ then $\frac{r_0}{3} \leq |x| \leq 3$. Using this in \eqref{EqIL}, we obtain
\begin{align*}
    I_L 
    & \gtrsim \tau \int_{|x-\tilde{x}_0| < \frac{r_0}{2\gamma}}  \frac{\mathrm{e}^{\tau (\log |x|)^2}}{|x|^2}   \dx  \gtrsim \tau A^{2}  \mathrm{e}^{\tau (\log (\frac{A}{2}))^2}   \Big( \frac{1}{\gamma^nA^nR^n}\Big)     \end{align*}
 for $A, R \gg1$. Observe that for  $A, R \gg1$, $|x-\tilde{x}_0|< \frac{r_0}{2\gamma}$ implies $\frac{1}{2A} \leq |x| \leq \frac{2}{A} <1$ and we use this fact in the last estimate above.  
Using this lower bound of $I_L$ in \eqref{Est2}, we obtain
\begin{align} \label{Est3}
  A^{2} R^{2-n}  \mathrm{e}^{\tau  (\log(\frac{A}{2}))^2}   \lesssim  
    R^{2-n} M_R(x_0)^2 \ \mathrm{e}^{\tau (\log (4 AR))^2} + A^{n-2} \mathrm{e}^{\tau (\log 4)^2} \,,  
\end{align}
where we use the fact that $\tau \sim A^2 R^{2}$. We can choose a large enough $A\gg1$ independent of $R$, $x_0$ and $\tau \sim A^2 R^{2} $ in order to absorb the last term in \eqref{Est3} to the left hand side of \eqref{Est3}.  Since the resulting inequality holds for any $x_0$ and $R\gg1$ with $|x_0|=R$, it follows that
$$M_R[\mathbb{U}] \gtrsim  \mathrm{e}^{-\kappa R^{2} (\log R)^{2}}$$
for some $\kappa >0$.
\end{proof}


\subsection{Improved quantitative estimates in special situations} \, \\
In this subsection, we are interested to investigate certain class of $\mathbb{V}, \mathbb{U}$ for which we have an improved vanishing order estimate of $M_R[\mathbb{U}]$ than \eqref{Est:Thm1}.    We will stick to dimension $n=2$. Of course, one can assume sufficient smoothness on $\mathbb{V}$ and transform the Dirac equation \eqref{Eq:Dirac} to a Schr\"odinger equation with complex drift. Then it is possible to deduce  improved estimates of $M_R[\mathbb{U}]$ from \cite{DKW22} under additional condition (other than smoothness) on $\mathbb{V}$. In the same way, we can get improved  estimates for real-valued $\mathbb{U}$ using \cite{LMNN20}. 
However, our goal is to avoid additional smoothness assumption on $\mathbb{V}$. Instead, we investigate whether certain structural conditions of $\mathbb{V}$ provide us a better estimate of $M_R[\mathbb{U}]$. 

\subsubsection{{Two particular cases of potentials}}

Towards this, first note that the Dirac equation \eqref{Eq:Dirac} in two dimension reduces to a system of ODE as follows 
\begin{align} \label{Eq:sys}
 i \bar{\partial} \mathbb{U}_1 & = V_{21} \mathbb{U}_1  +  V_{22} \mathbb{U}_2 \nonumber \\
    i \partial \mathbb{U}_2 &= V_{11} \mathbb{U}_1  +  V_{12} \mathbb{U}_2 \,,
\end{align}
where $\mathbb{U}=(\mathbb{U}_1,\mathbb{U}_2)$ and $\partial:= \partial_x-i\partial_y$ and $\bar{\partial}:= \partial_x+i\partial_y$. We consider two {particular} cases.

{\bf Case-1.} Let $V_{11}=0=V_{22}$. Let  $\mathbb{U} \in H^1_{\operatorname{loc}}(\R^2,\mathbb{C}^{2})$ be a bounded solution of \eqref{Eq:Dirac} with $|\mathbb{U}(0)|= {1}$. Without loss of generality, let us assume $\mathbb{U}_1(0)=1$. This is possible, as  we can assume without loss of generality that $\mathbb{U}_1(0) \neq 0$ (otherwise $\mathbb{U}_2(0)\neq 0$, and we argue in the same way with $\bar{\mathbb{U}}_2$, which satisfies $\bar{\partial}\bar{\mathbb{U}}_2 = i\bar{V}_{12}\bar{\mathbb{U}}_2$), and in this case $\widetilde{\mathbb{U}}:=\frac{\mathbb{U}}{\mathbb{U}_1(0)}$ is also a solution of \eqref{Eq:Dirac}.   Then a direct application of Proposition \ref{Prop:KSW} on $\mathbb{U}_1$ gives us
\begin{align*}
    M_R[\mathbb{U}] = \inf_{|x|=R} \left[\sup_{B_1(x)} |\mathbb{U}| \right] 
    & \gtrsim \max \left\{ \inf_{|x|=R} \left[\sup_{B_1(x)} |\mathbb{U}_1| \right] , \inf_{|x|=R} \left[\sup_{B_1(x)} |\mathbb{U}_2| \right] \right\} \\
    & \gtrsim \mathrm{e}^{-\kappa R \log R} \ \ \mbox{with} \ \ \kappa= \kappa(\|\mathbb{V} \|_{L^{\infty}}) >0 \,,
\end{align*}
for $R\gg1$.

{\bf Case-2.} Let $-\bar{V}_{12} = V_{21}$ and $-\bar{V}_{11}=V_{22}$. Taking the complex conjugate of the second equation in \eqref{Eq:sys} and subtracting it from the first  equation gives
$$i \bar{\partial} (\mathbb{U}_1 + \bar{\mathbb{U}}_2) = V_{21} \mathbb{U}_1  +  V_{22} \mathbb{U}_2 - \bar{V}_{11} \bar{ \mathbb{U}}_1  -  \bar{V}_{12} \bar{\mathbb{U}}_2 \,.$$
Thus,
\begin{align*}
    i \bar{\partial} (\mathbb{U}_1 + \bar{\mathbb{U}}_2) &= V_{21} \left(\mathbb{U}_1+\bar{\mathbb{U}}_2\right)    + V_{22} \left(\bar{ \mathbb{U}}_1 +\mathbb{U}_2 \right)  \\
    &= V_{21} \left(\mathbb{U}_1+\bar{\mathbb{U}}_2\right)    + V_{22} \left(\overline{ \mathbb{U}_1 +\bar{\mathbb{U}}_2} \right) \\
    &= \left( V_{21}     + V_{22} \frac{\overline{ \mathbb{U}_1 +\bar{\mathbb{U}}_2}}{\mathbb{U}_1+\bar{\mathbb{U}}_2} \right) \left(\mathbb{U}_1+\bar{\mathbb{U}}_2\right) := W \left(\mathbb{U}_1+\bar{\mathbb{U}}_2\right) \,.
\end{align*}
Clearly, $W$ is bounded as $V_{ij}$ are so. However, $\mathbb{U}_1+\bar{\mathbb{U}}_2$ may vanish at the origin, or even identically (indeed, $\mathbb{U}_2=-\bar{\mathbb{U}}_1$ is compatible with \eqref{Eq:sys} under the above symmetry assumptions). Since
$$
|\mathbb{U}_1+\bar{\mathbb{U}}_2|^2+|\mathbb{U}_1-\bar{\mathbb{U}}_2|^2 = 2\big(|\mathbb{U}_1|^2+|\mathbb{U}_2|^2\big) = 2|\mathbb{U}|^2
$$
and $|\mathbb{U}(0)|=1$, we have either $|\mathbb{U}_1(0)+\bar{\mathbb{U}}_2(0)|\geq 1$ or $|\mathbb{U}_1(0)-\bar{\mathbb{U}}_2(0)|\geq 1$. Without loss of generality, we assume the former. Indeed, since $\sigma_3$ anticommutes with $\sigma_1$ and $\sigma_2$, we have $\mathcal{D}_2\sigma_3=-\sigma_3\mathcal{D}_2$, so $\sigma_3\mathbb{U}=(\mathbb{U}_1,-\mathbb{U}_2)$ solves \eqref{Eq:Dirac} with $\mathbb{V}$ replaced by
$$
-\sigma_3\mathbb{V}\sigma_3=\begin{pmatrix} -V_{11} & V_{12} \\ V_{21} & -V_{22}\end{pmatrix},
$$
which satisfies the same symmetry assumptions and has the same $L^\infty$ norm as $\mathbb{V}$; moreover, $|\sigma_3\mathbb{U}|=|\mathbb{U}|$. Then, setting $W:=V_{21}$ on $\{\mathbb{U}_1+\bar{\mathbb{U}}_2=0\}$, the function
$$
F:=\frac{\mathbb{U}_1+\bar{\mathbb{U}}_2}{\mathbb{U}_1(0)+\bar{\mathbb{U}}_2(0)}
$$
is bounded, satisfies $\bar{\partial}F=-iWF$ with $|W|\le |V_{21}|+|V_{22}|$, and $|F(0)|=1$. Moreover, $|F|\le |\mathbb{U}_1|+|\mathbb{U}_2|\le \sqrt{2}\,|\mathbb{U}|$. Hence, Proposition \ref{Prop:KSW} applied to $F$ ensures that any bounded solution $\mathbb{U} \in H^1_{\operatorname{loc}}(\R^2,\mathbb{C}^{2})$ of \eqref{Eq:Dirac} with $|\mathbb{U}(0)|={ 1}$ will satisfy
\begin{align*}
    M_R[\mathbb{U}] = \inf_{|x|=R} \Big[\sup_{B_1(x)} |\mathbb{U}| \Big]  \geq \frac{1}{\sqrt{2}} \inf_{|x|=R} \Big[\sup_{B_1(x)} |F| \Big] \gtrsim \mathrm{e}^{-\kappa R \log R}
\end{align*}
with $\kappa= \kappa(\|\mathbb{V}_{L^{\infty}}) >0$, for $R\gg1$.

The above discussion leads to the following result.
\begin{theorem} \label{Thm:pot}
 Let $\mathbb{U} \in H^1_{\operatorname{loc}}(\R^2,\mathbb{C}^{2})$ be a bounded solution to \eqref{Eq:Dirac} such that $|\mathbb{U}(0)|={ 1}$ with either $V_{11}=0=V_{22}$, or $\bar{V}_{12} = -V_{21}$ and $V_{22}=-\bar{V}_{11}$. Then there exists $\kappa= \kappa(\|\mathbb{V} \|_{L^{\infty}})>0$ such that
 $$M_R[\mathbb{U}] \gtrsim  \mathrm{e}^{-\kappa R (\log R)}$$
 for $R\gg1$.   
\end{theorem}

\subsubsection{{Real-valued solutions.}}
Next we consider the real-valued case, i.e., the solution $\mathbb{U}= (\mathbb{U}_1,\mathbb{U}_2)$ is such that $\mathbb{U}_i(\cdot)\in \R$, $i=1,2$. {This situation arises for instance in the context of the \textit{Majorana equations}, which describe relativistic spin-$1/2$ fermions that are their own  \cite{Me37, We96, Sc14}. They have the same formal structure as the Dirac equation but impose the so called Majorana condition, identifying the field with its charge conjugate. This condition implies that, in a suitable choice  of representations, the spinor field can be taken real-valued \cite{Vin}, although in other representations its components are generally complex.} Recall that the Dirac equation \eqref{Eq:Dirac} reduces to the system of ODEs given by \eqref{Eq:sys}.
Define $$ F= \mathbb{U}_1 + i \mathbb{U}_2 \ \mbox{and} \ G=\mathbb{U}_1 - i \mathbb{U}_2 \,.$$ Then $G=\bar{F}$, as $\mathbb{U}$ is real-valued. Observe that $\bar \partial F = \bar \partial \mathbb{U}_1 + i \bar \partial \mathbb{U}_2 = \bar \partial \mathbb{U}_1 + i \,\overline{\partial \mathbb{U}_2}$. Thus
\begin{align*}
    \bar \partial F & =  -iV_{21} \mathbb{U}_1  -i  V_{22} \mathbb{U}_2 -  ( \bar{V}_{11} \mathbb{U}_1  +   \bar{V}_{12} \mathbb{U}_2) \\ &= -(\bar{V}_{11} + iV_{21}) \mathbb{U}_1 - (\bar{V}_{12} + i  V_{22}) \mathbb{U}_2 \\
    & = -(\bar{V}_{11} + iV_{21}) \left(\frac{F+G}{2} \right) - (\bar{V}_{12} + i  V_{22}) \left(\frac{F-G}{2i} \right) \\
    & = -\left(\frac{\bar{V}_{11}+ V_{22} + i[V_{21}-\bar{V}_{12}]}{2} \right) F  - \left(\frac{\bar{V}_{11}- V_{22} + i[V_{21}+\bar{V}_{12}]}{2} \right) \bar F = W F \,,
\end{align*}
where $$W=-\left(\frac{\bar{V}_{11}+ V_{22} + i[V_{21}-\bar{V}_{12}]}{2} \right) - \left(\frac{\bar{V}_{11}- V_{22} + i[V_{21}+\bar{V}_{12}]}{2} \right) \frac{\bar F}{F}\,.$$ Note that $W$ is bounded. Hence, it follows from \cite[Theorem 2.1]{KSW15} that any bounded real-valued solution $\mathbb{U} \in H^1_{\operatorname{loc}}(\R^2,\mathbb{C}^{2})$ of \eqref{Eq:Dirac} with $|\mathbb{U}(0)|={ 1}$ will satisfy
\begin{align*}
    M_R[\mathbb{U}] = \inf_{|x|=R} \left[\sup_{B_1(x)} |\mathbb{U}| \right] & =  \inf_{|x|=R} \left[\sup_{B_1(x)} |F| \right] \\ & \gtrsim 
 \mathrm{e}^{-\kappa R \log R} \ \ \mbox{with} \ \ \kappa= \kappa(\|\mathbb{V} \|_{L^{\infty}}) >0 \,,
\end{align*}

From the above discussion we obtain the following result.
\begin{theorem} \label{Thm:Real}
 Let $\mathbb{U} \in H^1_{\operatorname{loc}}(\R^2)$ be a bounded real-valued solution to \eqref{Eq:Dirac} such that $|\mathbb{U}(0)|= 1$. Then there exists $\kappa= \kappa(\|\mathbb{V} \|_{L^{\infty}})>0$ such that
 $$M_R \gtrsim  \mathrm{e}^{-\kappa R (\log R)}$$
 for $R\gg1$.   
\end{theorem}
\begin{remark} \rm
Note that $\mathbb{V}$ can be complex-valued in the above theorem and we do not assume any symmetry assumption on $\mathbb{V}$. For sufficient conditions for the Dirac equation \eqref{Eq:Dirac} to have real-valued solutions, we refer to \cite{Vin}.   
\end{remark}


\subsection{Qualitative Landis-type results with decaying potential}
In this subsection, we consider that the bounded potential $\mathbb{V}$ has some decay, say, $|\mathbb{V}(x)| \lesssim |x|^{-\la}$ near infinity for some $\lambda >0$. We are interested in deriving a qualitative unique continuation results capturing this decay behavior of the potential. In \cite{Ca22}, Cassano investigated the case $\lambda <1$.  He used the standard idea of using a suitable Carleman inequality to prove that a solution $\mathbb{U}$ of \eqref{Eq:Dirac} is compactly supported if  $\mathrm{e}^{k|x|^{2-2\lambda}} \mathbb{U} \in L^2(\R^n \setminus K ; \mathbb{C}^N)$ for some compact set $K$. Precisely, the following Carleman inequality was used:
\begin{align} \label{Eq:CI2}
    \tau  \int_{\R^n} \left[ \left(\partial_r^2 + \frac{1}{r} \partial_r\right) b(r) \right]    \mathrm{e}^{2\tau b(|x|)} |u(x)|^2 \dx \leq \int_{\R^n}\mathrm{e}^{2\tau b(|x|)} |\mathcal{D}_n u|^2 \dx
\end{align}
 for all $u \in C^{\infty}_c(\R^n \setminus \{0\},\mathbb{C}^{N})$ with the Carleman weight $b(r)=r^{a}$, $a>0$. Because of this particular choice of Carleman weight, the crucial restriction $\lambda <1$ appears in \cite[{Theorem 1.1}]{Ca22}.   However, the above Carleman inequality holds for any smooth $b$ in $\R^n \setminus \{0\}$. We choose the Carleman weight
 \begin{align} \label{Eq:b}
     b(|x|)=\log(1+|x|^a) \,, \ \  a>0 \,.
 \end{align}
 Then one can see that $$\left(\partial_r^2 + \frac{1}{r} \partial_r\right) b(r)= \frac{a^2}{r^{2-a}(1+r^{a})^2} \,.$$
 We observe that with this choice of Carleman weight, we can repeat the arguments of the proof of  \cite[Theorem 1.1]{Ca22} to treat the case $\lambda >1$. 
\begin{theorem} \label{Thm:qual1}
    Let $\mathbb{U} \in H^1_{\operatorname{loc}}(\R^n,\mathbb{C}^{N})$ be a solution to \eqref{Eq:Dirac} with bounded matrix-valued potential  $\mathbb{V}: \R^n \rightarrow  \mathbb{C}^{N \times N}$ which  decays as $|\mathbb{V}(x)| \lesssim |x|^{-\lambda}$ near infinity for some $\lambda >1$. Assume that there exists a compact set $K$ in $\R^n$ such that $|x|^{\kappa}\mathbb{U} \in L^2(\R^n\setminus K; \mathbb{C}^N)$ for all $\kappa>0$. Then $\mathbb{U} \equiv 0$ in $\mathbb{\R}^n$. 
\end{theorem}
\begin{proof}
We follow the lines of the proof of \cite[Theorem 1.1]{Ca22}. Let $a:=2(\lambda-1)>0$ and let $b$ be as in \eqref{Eq:b}. Fix $\rho\geq 1$ and $C_0>0$ such that $K\subset B_\rho$ and $|\mathbb{V}(x)|\le C_0|x|^{-\lambda}$ for $|x|\geq\rho$. Since $(1+|x|^a)^2\le 4|x|^{2a}$ and $2+a=2\lambda$, we have
$$
w(x):=\frac{a^2}{|x|^{2-a}(1+|x|^a)^2} \geq \frac{a^2}{4|x|^{2+a}} = \frac{a^2}{4|x|^{2\lambda}} \geq \frac{a^2}{4C_0^2}\,|\mathbb{V}(x)|^2 \,, \qquad |x|\geq \rho \,.
$$
Let $u\in H^1(\R^n;\mathbb{C}^N)$ with compact support in $\{|x|>\rho\}$. By \eqref{Eq:CI2}, which extends to such $u$ by density, and the above inequality,
\begin{align*}
\tau \int_{\R^n} w\, \mathrm{e}^{2\tau b(|x|)} |u|^2 \dx & \le 2\int_{\R^n} \mathrm{e}^{2\tau b(|x|)} |(\mathcal{D}_n+\mathbb{V})u|^2 \dx + 2\int_{\R^n} \mathrm{e}^{2\tau b(|x|)} |\mathbb{V}u|^2 \dx \\
& \le 2\int_{\R^n} \mathrm{e}^{2\tau b(|x|)} |(\mathcal{D}_n+\mathbb{V})u|^2 \dx + \frac{8C_0^2}{a^2}\int_{\R^n} w\, \mathrm{e}^{2\tau b(|x|)} |u|^2 \dx \,.
\end{align*}
Hence, for $\tau\geq \tau_0:=\frac{16C_0^2}{a^2}$, we can absorb the last term in the left-hand side and obtain
\begin{align} \label{Eq:CI3}
\tau \int_{\R^n} w\, \mathrm{e}^{2\tau b(|x|)} |u|^2 \dx \le 4\int_{\R^n} \mathrm{e}^{2\tau b(|x|)} |(\mathcal{D}_n+\mathbb{V})u|^2 \dx \,.
\end{align}
Let $h \in C^{\infty}(\R)$ be such that $h(t)=0$ for $t \leq 1$ and $h(t)=1$ for $t \geq 2$, and let $k \in C_c^{\infty}(\R)$ be such that $k(t)=1$ for $|t| \leq 1$ and $k(t)=0$ for $|t| \geq 2$. For $R>4\rho$, we apply \eqref{Eq:CI3} to
$$
u_R(x):= h\bigg(\frac{|x|}{\rho}\bigg) k\bigg(\frac{|x|}{R}\bigg) \mathbb{U}(x)\,.
$$
Since $(\mathcal{D}_n+\mathbb{V})\mathbb{U}=0$, we have $(\mathcal{D}_n+\mathbb{V})u_R=-i\big(\alpha\cdot\nabla \big[h(\frac{|x|}{\rho})k(\frac{|x|}{R})\big]\big)\mathbb{U}$, and since $|(\alpha\cdot\xi) z|=|\xi||z|$ for $\xi\in\R^n$, $z\in\mathbb{C}^N$,
$$
|(\mathcal{D}_n+\mathbb{V})u_R| \lesssim \frac{1}{\rho}\,|\mathbb{U}|\,\chi_{\{\rho<|x|<2\rho\}} + \frac{1}{R}\,|\mathbb{U}|\,\chi_{\{R<|x|<2R\}} \,.
$$
Recall that $b$ is increasing and $\mathrm{e}^{2\tau b(|x|)}=(1+|x|^a)^{2\tau}\le 2^{2\tau}|x|^{2\tau a}$ for $|x|\geq 1$. Then
$$
\tau \int_{2\rho<|x|<R} w\, \mathrm{e}^{2\tau b(|x|)} |\mathbb{U}|^2 \dx \lesssim \frac{\mathrm{e}^{2\tau b(2\rho)}}{\rho^2} \int_{\rho<|x|<2\rho} |\mathbb{U}|^2\dx + \frac{2^{2\tau}}{R^2}\int_{|x|>R} |x|^{2\tau a}|\mathbb{U}|^2 \dx \,,
$$
where the implicit constant depends only on $h$ and $k$. By assumption (with $\kappa=\tau a$), the last term goes to $0$ as $R\to\infty$. Letting $R\to\infty$, restricting the integral in the left-hand side to $\{|x|>4\rho\}$ and using $b(|x|)\geq b(4\rho)$ there, we get
$$
\int_{|x|>4\rho} w\, |\mathbb{U}|^2 \dx \lesssim \frac{1}{\tau\rho^2}\Big(\frac{1+(2\rho)^a}{1+(4\rho)^a}\Big)^{2\tau} \int_{\rho<|x|<2\rho} |\mathbb{U}|^2\dx
$$
for all $\tau\geq\tau_0$. Since the quotient in brackets is smaller than $1$, letting $\tau\to\infty$ we obtain $\mathbb{U}= 0$ in $\{|x|>4\rho\}$. Finally, $\mathbb{U} \equiv 0$ in $\R^n$ by the unique continuation property of the Dirac operator \cite{HP76}.
\end{proof}
\begin{remark}
$(i)$ Notice that $|\mathbb{V}(x)| \lesssim |x|^{-\lambda}$ near infinity for some $\lambda >1$ implies $|\mathbb{V}(x)| \lesssim |x|^{-\epsilon}$ near infinity for any $\epsilon <1$. Hence, \cite[{Theorem 1.1}]{Ca22} is still applicable for such potentials. However, according to \cite[{Theorem 1.1}]{Ca22}, we would require exponential decay on $\mathbb{U}$ to ensure $\mathbb{U} \equiv 0$. Whereas, in the above theorem we see that decay faster than any polynomial is sufficient to ensure this.  

  $(ii)$ Theorem \ref{Thm:qual1} can also be deduced from \cite[Theorem 1.7]{Ca25}. Indeed, since $\lambda>1$, we have $|\mathcal{D}_n\mathbb{U}|=|\mathbb{V}\mathbb{U}|\le \frac{1}{4|x|}|\mathbb{U}|$ for $|x|$ large, and the assumption of Theorem \ref{Thm:qual1} is equivalent to $\lim_{R\to\infty}R^k\int_{|x|>R}|\mathbb{U}|^2\dx=0$ for all $k\in\N$. Thus \cite[Theorem 1.7]{Ca25} implies that $\mathbb{U}$ has compact support, and then $\mathbb{U}\equiv0$ by \cite{HP76}. The proof above is direct, and it is based on the Carleman inequality \eqref{Eq:CI2} with the weight \eqref{Eq:b}.
\end{remark}

A natural question remains, i.e., what happens for $\lambda =1$? This corresponds to the critical potentials for the Dirac operators from the point of view of scaling. In this case, a qualitative Landis-type result in the exterior domain $\Om$ follows from a recent result  by Cassano, see \cite[Theorem 1.7]{Ca25}, under the assumption that $\||\cdot|V\|_{L^{\infty}(\Om)} < \frac{1}{2}$. More precisely, if  $|\mathbb{V}(x)| \leq \frac{C}{|x|}$ in $\Om:= \{x \in \R^n: |x|> \rho\}$ for some $\rho>0$ and some $C<\frac{1}{2}$, and 
$$
\lim_{R \rightarrow \infty} R^k \int_{|x|>R} |\mathbb{U}(x)|^2 \dx =0 \,, \ \ \forall k \in \mathbb{N} \,,
$$
then $\mathbb{U}\equiv 0$ in $\Om$. Moreover, $\frac{1}{2}$ is the optimal threshold in the sense that for each $C> \frac{1}{2}$ there are potentials $\mathbb{V}$ and nontrivial solution $\mathbb{U}$ such that $$ \||\cdot|V\|_{L^{\infty}(\Om)} < C $$ and $\mathbb{U}$ satisfies the above decay behavior. We note that the author used the Kelvin transformation to prove this qualitative result which creates a difficulty to extend the same for entire $\R^n$.     Our aim is to investigate this case with a very specific kind of potentials $\mathbb{V}$ of the form $\mathbb{V}= V I_{N \times N}$, where $V:\R^n \rightarrow \mathbb{C}$ is a radial function and $$ |V(x)| \lesssim |x|^{-1} \ \mbox{in} \ \R^n \,.$$ We emphasize that $V$ does not need to be bounded, i.e., for example the Coulomb potentials $V(x)= \alpha |x|^{-1}$ are included in this family for all $\alpha \in \R$. 

Observe that  the strong unique continuation result \cite[Theorem 2.1]{DO99}, see also \cite{KY99}, holds for {\it spherically symmetric scalar-type potentials} $\mathbb{V}(x)=V(|x|)I_{N\times N}$  without the smallness assumption on $|\cdot||V(\cdot)|$. Below we prove the unique continuation at infinity for such potentials.  Our proof of the result follows by careful inspection of the proof of \cite[Theorem 2.1]{DO99}.

\begin{theorem} \label{Thm:qual2}
    Let $\mathbb{U} \in H^1_{\operatorname{loc}}(\R^n,\mathbb{C}^{N})$ be a solution to \eqref{Eq:Dirac} where the potential is of the form $\mathbb{V}= V I_{N \times N}$, where $V:\R^n \rightarrow \mathbb{C}$ is a radial function and $|V(x)| \lesssim |x|^{-1}$ in $\R^n$.  Assume that
    $$\lim_{R \rightarrow \infty} R^{k} \int_{|x| \geq  R}  |\mathbb{U}(x)|^2 \dx =0 $$
for all $k \in \N$. Then $\mathbb{U} \equiv 0$ in $\mathbb{\R}^n$. 
\end{theorem}

The following polar decomposition of the Dirac operator will play an important role in the proof of above theorem. For $n \geq 2$, using the polar coordinates:
$$r=|x|, \qquad \om= \frac{x}{|x|}, \qquad y=\log r\,, $$
we can express the Dirac operator in terms of the operators acting on the sphere $\mathbb{S}^{n-1}$ as follows:
\begin{align} \label{Eq:Fourier}
    \mathcal{D}_n = \mathrm{e}^{-y} A \Big[ \partial_y + \big(\frac{n-1}{2}-B\big)\Big] \,,
\end{align}
where $$A(\om):= -i\sum_{j=1}^n \om_j \alpha_j $$ 
and 
$$
B(\om)= \sum_{j<k} \alpha_j \alpha_k \left[\om_k \frac{\partial}{\partial\om_j} - \om_j \frac{\partial}{\partial\om_k} \right] + \frac{n-1}{2}I_{N \times N} \,.
$$ 
We also have $AB=-BA$. It is known that the set of all eigenvalues of $B$ is $\sigma(B):= -\left(\mathbb{N}_0 + \frac{n-1}{2}\right) \cup \left(\mathbb{N}_0 + \frac{n-1}{2}\right)$, where $\mathbb{N}_0:=\{0,1,2, \ldots\}$. Let $E_{\lambda}$ be the $d_{\lambda}$-dimensional eigenspace of $B$ corresponding to the eigenvalue $\lambda$, which is spanned by the linearly independent eigenfunction $\{v_{\lambda,l}\}_{l=1}^{d_{\lambda}}$. The set of all eigenfunctions $\{v_{\lambda,l}\}$ for $\lambda \in \sigma(B)$ and $l=1, \ldots, d_{\lambda}$ forms an orthonormal basis of $L^2(\mathbb{S}^{n -1})^N$. Since $A^*A=I_{N \times N}$, $A^*=-A$ and $AB=-BA$, the operator $A$ maps $E_\lambda$ isometrically onto $E_{-\lambda}$. In particular, $d_\lambda=d_{-\lambda}$, and the basis can be chosen so that
$$A v_{\lambda,l} = v_{-\lambda,l}\,, \qquad A v_{-\lambda,l} = A^2 v_{\lambda,l}= -v_{\lambda,l}\,, \qquad \lambda>0\,, \ \ l=1,\ldots,d_{\lambda}\,.$$
We write this as $A v_{\lambda,l} = s_\lambda v_{-\lambda,l}$ for all $\lambda\in\sigma(B)$, where $s_\lambda:=1$ if $\lambda>0$ and $s_\lambda:=-1$ if $\lambda<0$. For details on the above decomposition \eqref{Eq:Fourier} of the Dirac operator and its properties, we refer to \cite[Section 4]{DO99}, see also \cite{J86}, where this decomposition was used to prove Carleman inequalities for the Dirac operator.

Now, a solution $\mathbb{U} \in L^2(\mathbb{R}^{n};\mathbb{C}^N)$ of \eqref{Eq:Dirac} can be written in the polar coordinate as follows:
\begin{align} \label{Fourier}
    \mathbb{U}(\mathrm{e}^{y},\om) = \sum_{\lambda \in \sigma(B)} \sum_{l=1}^{d_{\lambda}} f_{\lambda,l}(\mathrm{e}^{y}) v_{\lambda,l} (\omega) \,,
\end{align}
where $f_{\lambda,l}$ is determined by
$$f_{\lambda,l}(\mathrm{e}^y)= \langle \mathbb{U} , v_{\lambda,l}\rangle_{L^2(\mathbb{S}^{n-1})^N} = \int_{\mathbb{S}^{n-1}} \langle \mathbb{U}(\mathrm{e}^y,\om), v_{\lambda,l}(\om) \rangle_{\mathbb{C}^N} \d\om  \,.$$
Then [cf. \cite[(4.2)]{DO99}]
\begin{align*}
    \mathcal{D}_n \mathbb{U} (\mathrm{e}^y,\om) &= \mathrm{e}^{-y} A \Big[ \partial_y + \big(\frac{n-1}{2}-B\big)\Big] \mathbb{U} \\
    & = \mathrm{e}^{-y} \sum_{\lambda \in \sigma(B)} \sum_{l=1}^{d_{\lambda}}  s_{-\lambda}\Big[ \partial_y f_{-\lambda,l}(\mathrm{e}^y) + \big(\frac{n-1}{2}+\lambda\big) f_{-\lambda,l}(\mathrm{e}^y)\Big] v_{\lambda,l} (\omega) \,.
\end{align*}
As $\mathcal{D}_n \mathbb{U} = -\mathbb{V} \mathbb{U}$, we get [cf. \cite[(4.3)]{DO99}]
\begin{align} \label{Eq:ode}
   \Big[ \partial_y  + \big(\frac{n-1}{2}+\lambda\big) \Big] \widetilde{f}_{-\lambda,l}(y) = s_{-\lambda}\, \mathrm{e}^{y} \widetilde{V}(y) \widetilde{f}_{\lambda,l}(y) \,,
\end{align}
where $\widetilde{f}_{\lambda,l}(y):=f_{\lambda,l}(\mathrm{e}^{y})$ and $\widetilde{V}(y):= -V(\mathrm{e}^{y})$.

In order to prove Theorem \ref{Thm:qual2}, we will use the following one-dimensional Carleman inequality, {see \cite[Lemma 4.3]{DO99}}.
\begin{proposition} \label{Thm:CI-1}
For any $\nu >0$, we have
$$ 
\nu^2 \int_{\R} \mathrm{e}^{2\nu y} \varphi(y)^2 \ \dy \leq \int_{\R} \mathrm{e}^{2\nu y} [\partial_y \varphi(y)]^2 \ \dy  
$$
for all $\varphi \in C_c^{\infty}(\mathbb{R})$. By density, the same holds for compactly supported $\varphi\in H^1(\R;\mathbb{C})$, with $\varphi^2$ and $[\partial_y\varphi]^2$ replaced by $|\varphi|^2$ and $|\partial_y\varphi|^2$.
\end{proposition}

\begin{proof}[Proof of Theorem \ref{Thm:qual2}]
Let us assume that $\mathbb{U} \in H^1_{\operatorname{loc}}(\R^n,\mathbb{C}^N)$ be a solution to \eqref{Eq:Dirac} such that 
$$\lim_{R \rightarrow \infty} R^{k} \int_{|x| \geq  R}  |\mathbb{U}(x)|^2 \dx = 0 $$
for all $k \in \N$. This implies in particular that $\mathbb{U} \in L^2(\R^n,\mathbb{C}^N)$, and hence, we can decompose $\mathbb{U}$ as in \eqref{Fourier}.  We take a cut-off function $\eta_R \in C_c^{\infty}(\R)$ such that $\eta_R \equiv 1$ in $(-\log R, \log R)$ and $0$ outside $(-2\log R,2\log R)$, $0\leq \eta_R \leq 1$, and $|\eta_R'(y)| \lesssim (\log R)^{-1}$. 
Fix $\lambda\in\sigma(B)$ and $l\in\{1,\ldots,d_\lambda\}$, and write $g_{\pm}:=\widetilde{f}_{\pm\lambda,l}$. By \eqref{Eq:ode}, applied with $\lambda$ and with $-\lambda$,
$$
\partial_y g_{\mp} = -\Big(\frac{n-1}{2}\pm\lambda\Big) g_{\mp} + s_{\mp\lambda}\,\mathrm{e}^{y}\widetilde{V}(y)\, g_{\pm} \,.
$$
Since $|V(x)|\le C_V|x|^{-1}$ for some $C_V>0$, we have $\mathrm{e}^{y}|\widetilde{V}(y)|\le C_V$. Hence, with $c_\lambda:=\frac{n-1}{2}+|\lambda|$,
$$
|\partial_y g_{\pm}|\le c_\lambda |g_{\pm}| + C_V |g_{\mp}| \,.
$$
Applying Proposition \ref{Thm:CI-1} to $\eta_R g_{\pm}$, we obtain
\begin{align*}
\nu^2 \int_{\R} \mathrm{e}^{2\nu y} \eta_R^2 |g_{\pm}|^2 \dy & \le \int_{\R} \mathrm{e}^{2\nu y} \big|\eta_R\, \partial_y g_{\pm} + g_{\pm}\, \partial_y\eta_R\big|^2 \dy \\
& \le 3\int_{\R} \mathrm{e}^{2\nu y} \eta_R^2 \big(c_\lambda^2 |g_{\pm}|^2 + C_V^2 |g_{\mp}|^2\big) \dy + 3\int_{\R} \mathrm{e}^{2\nu y} |\partial_y \eta_R|^2 |g_{\pm}|^2 \dy \,.
\end{align*}
Adding the two inequalities and setting
$$
I_R:=\int_{\R} \mathrm{e}^{2\nu y} \eta_R^2 \big(|g_+|^2+|g_-|^2\big) \dy <\infty \,,
$$
we get
$$
\nu^2 I_R \le 3\big(c_\lambda^2+C_V^2\big) I_R + 3\int_{\R} \mathrm{e}^{2\nu y} |\partial_y \eta_R|^2 \big(|g_+|^2+|g_-|^2\big) \dy \,.
$$
We choose $\nu>n$ such that $\nu^2\geq 6(c_\lambda^2+C_V^2)$. Then we can absorb the first term on the right-hand side in the left-hand side. Moreover, by Bessel's inequality (recall that $\{v_{\lambda,l}\}$ is orthonormal),
$$
|g_+(y)|^2+|g_-(y)|^2 \le m(y):=\int_{\mathbb{S}^{n-1}} |\mathbb{U}(\mathrm{e}^y,\om)|^2 \d\om \,.
$$
Therefore,
$$
\nu^2 I_R \lesssim \frac{1}{(\log R)^2} \left[ \int_{-2\log R}^{-\log R} \mathrm{e}^{2\nu y}\, m(y)  \dy + \int_{\log R}^{2\log R} \mathrm{e}^{2\nu y}\, m(y)  \dy \right] .
$$
We estimate the two integrals separately, using $\dx=\mathrm{e}^{ny}\dy\,\d\om$. For $y\in[-2\log R,-\log R]$, we have $\mathrm{e}^{2\nu y}\le R^{-2\nu}$ and $\mathrm{e}^{-ny}\le R^{2n}$, so
$$
\int_{-2\log R}^{-\log R} \mathrm{e}^{2\nu y}\, m(y)  \dy \le R^{-2\nu} \int_{R^{-2}<|x|<R^{-1}} \frac{|\mathbb{U}(x)|^2}{|x|^n} \dx \le R^{2n-2\nu} \int_{|x|<1} |\mathbb{U}|^2 \dx \,.
$$
For $y\in[\log R,2\log R]$, we have $\mathrm{e}^{2\nu y}\le R^{4\nu}$ and $\mathrm{e}^{-ny}\le R^{-n}$, so
$$
\int_{\log R}^{2\log R} \mathrm{e}^{2\nu y}\, m(y)  \dy \le R^{4\nu-n} \int_{|x|\geq R} |\mathbb{U}|^2 \dx \,.
$$
Thus,
$$
\nu^2 I_R \lesssim \frac{1}{(\log R)^2} \left[ R^{2n-2\nu} \int_{|x|<1} |\mathbb{U}|^2 \dx + R^{4\nu-n} \int_{|x|\geq R} |\mathbb{U}|^2 \dx \right] \longrightarrow 0 \quad \mbox{as} \ \ R\to\infty \,,
$$
since $\nu>n$ and by the decay assumption on $\mathbb{U}$ (with $k\geq 4\nu$). As $\eta_R\to 1$ pointwise when $R\to\infty$, Fatou's lemma gives
$$
\int_{\R} \mathrm{e}^{2\nu y} \big(|g_+|^2+|g_-|^2\big) \dy \le \liminf_{R\to\infty} I_R = 0 \,,
$$
that is, $\widetilde{f}_{\lambda,l}\equiv \widetilde{f}_{-\lambda,l}\equiv 0$. Since $\lambda\in\sigma(B)$ and $l$ are arbitrary, $\mathbb{U} \equiv 0$ in $\R^n$.
\end{proof}

\begin{remark} \rm Observe that for the Coulomb potentials of the form $\mathbb{V}(x)= \frac{\alpha}{|x|}I_{N \times N}$, the result in \cite[Theorem 1.7]{Ca25} is applicable only in the exterior domain and for $|\alpha| < \frac{1}{2}$. Notably, in the above theorem we do not get any smallness restriction on the values of $\alpha$ and the result works in whole $\R^n$. This happens due to the spherically symmetric scalar-type structure of the potential $\mathbb{V}$.   
    
\end{remark}

\medskip
\noindent{\textbf{Acknowledgments}.}
{The authors are grateful to Biagio Cassano for fruitful discussions on the qualitative results of the present manuscript. U. Das, L. Fanelli and L. Roncal are supported 
by the Spanish Agencia Estatal de Investigaci\'{o}n
through BCAM Severo Ochoa CEX2021-001142-S/MCIN/AEI/10.13039/501100011033 and by the Basque Government through the BERC 2026--2029 program. U. Das and L. Roncal are also partially suppported by 
CNS2023-143893 funded by MICIU/AEI/10.13039/501100011033 and by European Union NextGenerationEU/PRTR. L. Fanelli and L. Roncal are also supported by IKERBASQUE. L. Fanelli is also supported by the project PID2024-155550NB-100 funded by MICIU/AEI/10.13039/501100011033 and FEDER/EU and IT1875-26 under the Basque Government. 
L. Roncal is also partially supported by PID2023-146646NB-I00 funded by MICIU/AEI/10.13039/501100011033 and FEDER/EU and by ESF+.}


\end{document}